\documentclass[oneside,english]{amsart}
\usepackage[T1]{fontenc}
\usepackage[latin9]{inputenc}
\usepackage{geometry}
\usepackage{color}
\usepackage{babel}
\usepackage{amstext}
\usepackage{amsthm}
\usepackage{amssymb}
\usepackage{stmaryrd}
\usepackage[unicode=true,pdfusetitle,
 bookmarks=true,bookmarksnumbered=false,bookmarksopen=false,
 breaklinks=false,pdfborder={0 0 0},pdfborderstyle={},backref=false,colorlinks=true]
 {hyperref}

\makeatletter
\numberwithin{equation}{section}
\numberwithin{figure}{section}
 \theoremstyle{definition}
 \newtheorem*{defn*}{\protect\definitionname}
  \theoremstyle{plain}
  \newtheorem*{thm*}{\protect\theoremname}
  \theoremstyle{plain}
  \newtheorem*{cor*}{\protect\corollaryname}
\theoremstyle{plain}
\newtheorem{thm}{\protect\theoremname}
  \theoremstyle{plain}
  \newtheorem{prop}[thm]{\protect\propositionname}
  \theoremstyle{plain}
  \newtheorem{lem}[thm]{\protect\lemmaname}
  \theoremstyle{plain}
  \newtheorem{cor}[thm]{\protect\corollaryname}
  \theoremstyle{definition}
  \newtheorem{example}[thm]{\protect\examplename}

\usepackage[all]{xy}
\usepackage{graphics}
\usepackage{epsfig}
\usepackage{graphicx}
\usepackage {enumitem}

\makeatother

  \providecommand{\corollaryname}{Corollary}
  \providecommand{\definitionname}{Definition}
  \providecommand{\examplename}{Example}
  \providecommand{\lemmaname}{Lemma}
  \providecommand{\propositionname}{Proposition}
  \providecommand{\theoremname}{Theorem}
\providecommand{\theoremname}{Theorem}

\begin{document}
\title[Cylinders in Mori Fiber Spaces]{Cylinders in Mori Fiber Spaces: Forms of the quintic del Pezzo threefold}
\author{Adrien Dubouloz}
\address{IMB UMR5584, CNRS, Univ. Bourgogne Franche-Comté, F-21000 Dijon, France} \email{adrien.dubouloz@u-bourgogne.fr }
\author{Takashi Kishimoto}
\address{Department of Mathematics, Faculty of Science, Saitama University, Saitama 338-8570, Japan} \email{tkishimo@rimath.saitama-u.ac.jp}

\thanks{The second author was partially funded by Grant-in-Aid for Scientific Research of JSPS No. 15K04805. The research was initiated during a visit of the first author at Saitama University and continued during a stay of the second author at University of Burgundy as a CNRS Research Fellow. The authors thank these institutions for their generous support and the excellent working conditions offered.} 


\begin{abstract}
Motivated by the general question of existence of open $\mathbb{A}^{1}$-cylinders
in higher dimensional projective varieties, we consider the case of
Mori Fiber Spaces of relative dimension three, whose general closed
fibers are isomorphic to the quintic del Pezzo threefold $V_{5}$,
the smooth Fano threefold of index two and degree five. We show that
the total spaces of these Mori Fiber Spaces always contain relative
$\mathbb{A}^{2}$-cylinders, and we characterize those admitting relative
$\mathbb{A}^{3}$-cylinders in terms of the existence of certain special
lines in their generic fibers. 
\end{abstract}

\maketitle

\section*{Introduction}

An $\mathbb{A}_{k}^{r}$-\emph{cylinder} in a normal algebraic variety
defined over a field $k$ is a Zariski open subset $U$ isomorphic
to $Z\times\mathbb{A}_{k}^{r}$ for some algebraic variety $Z$ defined
over $k$. In the case where $k=\overline{k}$ is algebraically closed,
normal projective varieties $V$ containing $\mathbb{A}_{\overline{k}}^{1}$-cylinders
have received a lot of attention recently due to the connection between
unipotent group actions on their affine cones and polarized $\mathbb{A}_{\overline{k}}^{1}$-cylinders
in them, that is $\mathbb{A}_{\overline{k}}^{1}$-cylinders whose
complements are the supports of effective $\mathbb{Q}$-divisors linearly
equivalent to an ample divisor on $V$ (cf. \cite{KPZ1,KPZ2}). Certainly,
the canonical divisor $K_{V}$ of a normal projective variety containing
an $\mathbb{A}_{\overline{k}}^{r}$-cylinder for some $r\geq1$ is
not pseudo-effective. Replacing $V$ if necessary by a birational
model with at most $\mathbb{Q}$-factorial terminal singularities,
\cite{BCHM} guarantees the existence of a suitable $K_{V}$-MMP $V\dashrightarrow X$
whose output $X$ is equipped with a structure of Mori Fiber Space
$f:X\rightarrow Y$ over some lower dimensional normal projective
variety $Y$. Since an $\mathbb{A}_{\overline{k}}^{1}$-cylinder in
$X$ can always be transported back in the initial variety $V$ \cite[Lemma 9]{DK5},
total spaces of Mori Fiber Spaces form a natural restricted class
in which to search for varieties containing $\mathbb{A}_{\overline{k}}^{1}$-cylinders. 

In the case where $\dim Y=0$, $X$ is a Fano variety of Picard number
one. The only smooth Fano surface of Picard number is the projective
plane $\mathbb{P}_{\overline{k}}^{2}$ which obviously contains $\mathbb{A}_{\overline{k}}^{1}$-cylinders.
Several families of examples of smooth Fano varieties of dimension
$3$ and $4$ and Picard number one containing $\mathbb{A}_{\overline{k}}^{1}$-cylinders
have been constructed \cite{KPZ4,PZ14,PZ15}. The question of existence
of $\mathbb{A}_{\overline{k}}^{1}$-cylinders in other possible outputs
of MMPs was first considered in \cite{DK3,DK4}, in which del Pezzo
fibrations $f:X\rightarrow Y$, which correspond to the case where
$\dim Y=\dim X-2>0$, were extensively studied. In such a relative
context, it is natural to shift the focus to cylinders which are compatible
with the fibration structure: 
\begin{defn*}
Let $f:X\to Y$ be a morphism between normal algebraic varieties defined
over a field $k$ and let $U\simeq Z\times\mathbb{A}_{k}^{r}$ be
an $\mathbb{A}_{k}^{r}$-cylinder inside $X$. We say that $U$ is
\emph{vertical with respect to} $f$ if the restriction $f|_{U}$
factors as 
\[
f\mid_{U}=h\circ\mathrm{pr}_{Z}:U\simeq Z\times\mathbb{A}_{k}^{r}\stackrel{\mathrm{pr}_{Z}}{\longrightarrow}Z\stackrel{h}{\longrightarrow}Y
\]
for a suitable morphism $h:Z\rightarrow Y$. 
\end{defn*}
In the present article, we initiate the study of existence of vertical
$\mathbb{A}^{1}$-cylinders in Mori Fiber Spaces $f:X\rightarrow Y$
of relative dimension three, whose general fibers are smooth Fano
threefolds. Each smooth Fano threefold of Picard number one can appear
as a closed fiber of a Mori Fiber Space, and among these, it is natural
to first restrict to classes which contain a cylinder of the maximal
possible dimension, namely the affine space $\mathbb{A}_{\overline{k}}^{3}$,
and expect that some suitable sub-cylinders of these fiber wise maximal
cylinders could arrange into a vertical cylinder with respect to $f$,
possibly of smaller relative dimension. The only four classes of Fano
threefold of Picard number one containing $\mathbb{A}_{\overline{k}}^{3}$
are $\mathbb{P}_{\overline{k}}^{3}$, the quadric $\mathbb{Q}^{3}$,
the del Pezzo quintic threefold $V_{5}$ of index two and degree five
\cite{FuNa89,FuNa89-2}, and a four dimensional family of prime Fano
threefolds $V_{22}$ of genus twelve \cite{Fu93,Pr}. \\

The existence inside $X$ of a vertical $\mathbb{A}^{r}$-cylinder
with respect to $f:X\rightarrow Y$ translates equivalently into that
of an $\mathbb{A}_{K}^{r}$-cylinder inside the fiber $X_{\eta}$
of $f$ over the generic point $\eta$ of $Y$, considered as a variety
defined over the function field $K$ of $Y$ \cite{DK4}. We are therefore
led to study the existence of $\mathbb{A}_{K}^{r}$-cylinders inside
$K$-forms of the aforementioned Fano threefolds over non-closed fields
$K$ of characteristic zero, that is, smooth projective varieties
defined over $K$ whose base extensions to an algebraic closure $\overline{K}$
are isomorphic over $\overline{K}$ to one of these Fano threefolds.
The case of $\mathbb{P}^{3}$ is easily dispensed: a $K$-form $V$
of $\mathbb{P}^{3}$ contains an $\mathbb{A}_{K}^{3}$-cylinder if
and only if it has a $K$-rational point hence if and only if it is
the trivial $K$-form $\mathbb{P}_{K}^{3}$. So equivalently, a Mori
Fiber Space $f:X\rightarrow Y$ whose general closed fibers are isomorphic
to $\mathbb{P}_{\overline{k}}^{3}$ contains a vertical $\mathbb{A}_{\overline{k}}^{3}$-cylinder
if and only if it has a rational section. The case of the quadric
$\mathbb{Q}^{3}$ is already more intricate: one can deduce from \cite{Fu93}
that a $K$-form $V$ of $\mathbb{Q}^{3}$ contains $\mathbb{A}_{K}^{3}$
if and only if it has a hyperplane section defined over $K$ which
is a $K$-rational quadric cone. In this article, we establish a complete
characterization of the existence of $\mathbb{A}_{K}^{r}$-cylinders
in forms of $V_{5}$ which can be summarized as follows: 
\begin{thm*}
Let $K$ be a field of characteristic zero and let $Y$ be a $K$-form
of $V_{5}$. Then $Y$ always contains an $\mathbb{A}_{K}^{2}$-cylinder,
and it contains an $\mathbb{A}_{K}^{3}$-cylinder if and only if it
contains a curve $\ell\simeq\mathbb{P}_{K}^{1}$ of anticanonical
degree $-K_{Y}\cdot\ell=2$ and with normal bundle $\mathcal{N}_{\ell/Y}\simeq\mathcal{O}_{\mathbb{P}_{K}^{1}}(-1)\oplus\mathcal{O}_{\mathbb{P}_{K}^{1}}(1)$. 
\end{thm*}
An irreducible curve $\ell$ of anticanonical degree $-K_{Y}\cdot\ell=2$
on a $K$-form $Y$ of $V_{5}$ becomes after base extension to an
algebraic closure $\overline{K}$ of $K$ a usual line on $Y_{\overline{K}}\simeq V_{5}$
embedded into $\mathbb{P}_{\overline{K}}^{6}$ via its half-anticanonical
complete linear system. By combining the previous characterization
with a closer study of the Hilbert scheme of such curves $\ell$ on
$Y$ (see $\S$\ref{subsec:Lines-on-fom-Hilbert}), we derive the
following result (see Corollary \ref{cor:A3-overC1-field}):
\begin{cor*}
Let $\overline{k}$ be an algebraically closed field of characteristic
zero and let $f:X\rightarrow C$ be a Mori Fiber Space over a curve
$C$ defined over $\overline{k}$,  whose general closed fibers are
quintic del Pezzo threefolds $V_{5}$. Then $X$ contains a vertical
$\mathbb{A}_{\overline{k}}^{3}$-cylinder with respect to $f$. 
\end{cor*}
Section \ref{sec:1} contains a brief recollection on the quintic
del Pezzo threefold $V_{5}$ and its Hilbert scheme of lines. In Section
\ref{sec:2}, we establish basic geometric properties of forms of
$V_{5}$ and describe their Hilbert schemes of lines. We also describe
an adaptation to non-closed fields of a standard construction of $V_{5}$
as the variety of trisecant lines to a Veronese surface in $\mathbb{P}^{4}$,
from which we derive for suitable fields the existence of forms of
$V_{5}$ which do not contain any line with normal bundle $\mathcal{O}_{\mathbb{P}^{1}}(-1)\oplus\mathcal{O}_{\mathbb{P}^{1}}(1)$.
The technical core of the article is then Section \ref{sec:3}: we
give a new construction of a classical rational map, called the double
projection from a rational point of a form of $V_{5}$ in the form
of a Sarkisov link from a $\mathbb{P}^{1}$-bundle over $\mathbb{P}^{2}$
explicitly determined by the base locus of a quadratic birational
involution of $\mathbb{P}^{2}$. The main results concerning $\mathbb{A}^{r}$-cylinders
in forms of $V_{5}$ are then derived from this construction in Section
\ref{sec:4}. 

\section{\label{sec:1}Geometry of the smooth quintic del Pezzo threefold}

In the rest of this article, unless otherwise stated, the notations
$k$ and $\overline{k}$ refer respectively to a field of characteristic
zero and a fixed algebraic closure of $k$. In this section, we recall
without proof classical descriptions and properties of the quintic
del Pezzo threefold $V_{5}$ over $\overline{k}$ and of its Hilbert
scheme of lines. 

\subsection{Two classical descriptions of $V_{5}$ }

A quintic del Pezzo threefold $V_{5}$ over $\overline{k}$ is a smooth
projective threefold whose Picard group is isomorphic to $\mathbb{Z}$,
generated by an ample class $H$ such that $-K_{V_{5}}=2H$ and $H^{3}=5$.
In other words, $V_{5}$ is a smooth Fano threefold of index two and
degree five. 

\subsubsection{\label{subsec:Iskov-link} Sarkisov links to a smooth quadric in
$\mathbb{P}^{4}$ }

Let us first recall the classical description due to Iskovskikh \cite[Chapter II, \S 1.6]{Is80}.
Letting $H$ be an ample class such that $-K_{V_{5}}=2H$, the complete
linear system $|H|$ defines a closed embedding $\Phi_{|H|}:V_{5}\hookrightarrow\mathbb{P}_{\overline{k}}^{6}$.
A general hyperplane section of $V_{5}$ contains a line $\ell$ whose
normal bundle in $V_{5}$ is trivial, and the projection from $\ell$
induces a birational map $V_{5}\dashrightarrow Q$ onto a smooth quadric
$Q\subset\mathbb{P}_{\overline{k}}^{4}$. The inverse map $Q\dashrightarrow V_{5}$
can be described as the blow-up of $Q$ along a rational normal cubic
$C\subset Q$ contained in a smooth hyperplane section $Q_{0}\simeq\mathbb{P}_{\overline{k}}^{1}\times\mathbb{P}_{\overline{k}}^{1}$
of $Q$, followed by the contraction of the proper transform $Q_{0}'$
of $Q_{0}$ onto the line $\ell$. In sum, the projection from the
line $\ell$ induces a Sarkisov link \[\xymatrix@R=3em@C=2em{ & Q_0' \ar@{^{(}->}[rr] \ar[dl] & & \tilde{V}_5 \ar[dl]_{q} \ar[dr]^{q'} & & Z' \ar@{_{(}->}[ll] \ar[dr] \\ \ell \ar@{^{(}->}[rr] & & V_5 \ar@{-->}@/_10pt/[rr]_{|H-\ell|} & & Q & & C \ar@{_{(}->}[ll]}\]where
$q:\tilde{V}_{5}\rightarrow V_{5}$ is the blow-up of $V_{5}$ along
$\ell$ with exceptional divisor $Q_{0}'$ and $q':\tilde{V}_{5}\rightarrow Q$
is the contraction onto the curve $C$ of the proper transform $Z'$
of the surface $Z\subset V_{5}$ swept out by lines in $V_{5}$ intersecting
$\ell$. 

Since the automorphism group $\mathrm{PGL}_{5}(\overline{k})$ of
$\mathbb{P}_{\overline{k}}^{4}$ acts transitively on the set of flags
$C\subset Q_{0}\subset Q$, it follows that over an algebraically
closed field $\overline{k}$, all smooth Fano threefold of Picard
number one, index two and degree five embedded into $\mathbb{P}_{\overline{k}}^{6}$
by their half-anticanonical complete linear system are projectively
equivalent. 

\subsubsection{\label{subsec:QuasiHom} Quasi-homogeneous space of $\mathrm{PGL}_{2}(\overline{k})$}

We now recall an alternative description of $V_{5}$ due to Mukai-Umemura
\cite{MuUm83} (see also \cite[\S 5.1]{KPS18}). Let $M_{d}=\mathrm{Sym}^{d}(\overline{k}^{2})^{\vee}\simeq\overline{k}[x,y]_{(d)}$
be the space of homogeneous polynomials of degree $d$ with coefficients
in $\overline{k}$. The natural action of $\mathrm{GL}_{2}(\overline{k})$
on $M_{1}$ induces a linear action on $M_{6}$, hence an action of
$\mathrm{PGL}_{2}(\overline{k})$ on $\mathbb{P}(M_{6})\simeq\mathbb{P}_{\overline{k}}^{6}$.
We then have the following description:
\begin{prop}
The Fano threefold $V_{5}$ is isomorphic to the closure $\overline{\mathrm{PGL}_{2}(\overline{k})\cdot[\phi]}$
of the class of the polynomial $\phi=xy(x^{4}+y^{4})\in M_{6}$. Furthermore,
the $\mathrm{PGL}_{2}(\overline{k})$-orbits on $V_{5}$ are described
as follows: 

\begin{enumerate}[label=$($\roman*$)$]

\item  The open orbit $O=\mathrm{PGL}_{2}(\overline{k})\cdot[\phi]$
with stabilizer equal to the binary octahedral group,

\item  The $2$-dimensional orbit $S_{2}=\mathrm{PGL}_{2}(\overline{k})\cdot[xy^{5}]$,
which is neither open nor closed, with stabilizer equal to the diagonal
torus $\mathbb{T}$. 

\item  The $1$-dimensional closed orbit $C_{6}=\mathrm{PGL}_{2}(\overline{k})\cdot[x^{6}]$
with stabilizer equal to the Borel subgroup $B$ of upper triangular
matrices. 

\end{enumerate}
\end{prop}

It follows from this description that the automorphism $\mathrm{Aut}(V_{5})$
is isomorphic to $\mathrm{PGL}_{2}(\overline{k})$. We also observe
that $C_{6}$ is a normal rational sextic curve and that the closure
$\overline{S}_{2}=S_{2}\cup C_{6}$ of $S_{2}$ is a quadric section
of $V_{5}$, hence an anti-canonical divisor on $V_{5}$, which coincides
with the \emph{tangential scroll} of $C_{6}$, swept out by the tangent
lines to $C_{6}$ contained in $V_{5}$. It is singular along $C_{6}$,
and its normalization morphism coincides with the map $\nu:\mathbb{P}(M_{1})\times\mathbb{P}(M_{1})\rightarrow\mathbb{P}(M_{6})$,
$(f_{1},f_{2})\mapsto f_{1}^{5}f_{2}$.

\subsection{\label{subsec:Hilbert-Lines}Lines on $V_{5}$ }

The family of lines on $V_{5}$ is very well-studied \cite{Is80,FuNa89,Il94-2}.
We list below some of its properties which will be useful later on
for the study of cylinders on forms of $V_{5}$. 

First, by a \emph{line} on $V_{5}$, we mean an integral curve $\ell\subset V_{5}$
of anticanonical degree $-K_{V_{5}}\cdot\ell=2$. It thus corresponds
through the half-anticanonical embedding $\Phi_{|H|}:V_{5}\hookrightarrow\mathbb{P}_{\overline{k}}^{6}$
to a usual line in $\mathbb{P}_{\overline{k}}^{6}$ which is contained
in the image of $V_{5}$. A general line $\ell$ in $V_{5}$ has trivial
normal bundle, whereas there is a one-dimensional subfamily of lines
with normal bundle $\mathcal{N}_{\ell/V_{5}}\simeq\mathcal{O}_{\mathbb{P}_{\overline{k}}^{1}}(-1)\oplus\mathcal{O}_{\mathbb{P}_{\overline{k}}^{1}}(1)$,
which we call \emph{special lines}. 

The Hilbert scheme $\mathcal{H}(V_{5})$ of lines in $V_{5}$ is isomorphic
to $\mathbb{P}_{\overline{k}}^{2}$, and the evaluation map $\upsilon:\mathcal{U}\rightarrow V_{5}$
from the universal family $\mathcal{U}\rightarrow\mathcal{H}(V_{5})$
is a finite morphism of degree $3$. There are thus precisely three
lines counted with multiplicities passing through a given closed point
of $V_{5}$. 

The curve in $\mathcal{H}(V_{5})\simeq\mathbb{P}_{\overline{k}}^{2}$
that parametrizes special lines in $V_{5}$ is a smooth conic $C$.
The restriction of $\upsilon$ to $\mathcal{U}|_{C}$ is injective,
and in the description of $V_{5}$ as a quasi-homogeneous space of
$\mathrm{PGL}_{2}(\overline{k})$ given in $\S$ \ref{subsec:QuasiHom}
above, $\upsilon(\mathcal{U}|_{C})$ is the tangential scroll $\overline{S}_{2}$
to the rational normal sextic $C_{6}$, while $C_{6}$ itself coincides
with the image of the intersection of $\mathcal{U}|_{C}$ with the
ramification locus of $\upsilon$. In particular, special lines on
$V_{5}$ never intersect each others. Furthermore, there are three
lines with trivial normal bundle through any point in $V_{5}\setminus\overline{S}_{2}$,
a line with trivial normal bundle and a special line through any point
of $S_{2}$, and a unique special line through every point of $C_{6}$. 

\section{\label{sec:2}Forms of $V_{5}$ over non-closed fields}

A $k$\emph{-form of} $V_{5}$ is a smooth projective variety $Y$
defined over $k$ such that $Y_{\overline{k}}$ is isomorphic to $V_{5}$.
In this subsection, we establish basic properties of these forms and
their Hilbert schemes of lines. 

\subsection{\label{subsec:Lines-on-fom-Hilbert}Hilbert scheme of lines on a
$k$-form of $V_{5}$ }

Let $Y$ be a $k$-form of $V_{5}$, let $\mathcal{H}(Y)$ be the
Hilbert scheme of irreducible curves of $-K_{Y}$-degree equal to
$2$ on $Y$ and let $\upsilon:\mathcal{U}\rightarrow Y$ be the evaluation
map from the universal family $\mathcal{U}\rightarrow\mathcal{H}(Y)$.
By $\S$ \ref{subsec:Hilbert-Lines}, $\mathcal{H}(Y)_{\overline{k}}=\mathcal{H}(Y_{\overline{k}})$
is isomorphic to $\mathbb{P}_{\overline{k}}^{2}$, i.e. $\mathcal{H}(Y)$
is a $k$-form of $\mathbb{P}_{\overline{k}}^{2}$. Furthermore, since
the smooth conic parametrizing special lines on $Y_{\overline{k}}$
is invariant under the action of the Galois group $\mathrm{Gal}(\overline{k}/k)$,
it corresponds to a smooth curve $C\subset\mathcal{H}(Y)$ defined
over $k$, and such that $C_{\overline{k}}\simeq\mathbb{P}_{\overline{k}}^{1}$. 
\begin{lem}
\label{lem:Forms-General-Line}\label{lem:Existence-of-specialLines}
Let $Y$ be a $k$-form of $V_{5}$. 

a) The Hilbert scheme $\mathcal{H}(Y)$ of lines on $Y$ is isomorphic
to $\mathbb{P}_{k}^{2}$, in particular $Y$ always contains lines
defined over $k$ with trivial normal bundles. 

b) The following assertions are equivalent:

\begin{enumerate}[label=$($\roman*$)$]

\item $Y$ contains a special line defined over $k$, 

\item The conic $C$ has a $k$-rational point, 

\item The surface $\upsilon(\mathcal{U}|_{C})$ has a $k$-rational
point,

\item The image of the intersection of $\mathcal{U}|_{C}$ with the
ramification locus of $\upsilon$ has a $k$-rational point.

\end{enumerate}
\end{lem}

\begin{proof}
Since $\mathcal{H}(Y)$ is a $k$-form of $\mathbb{P}_{\overline{k}}^{2}$,
to prove the first assertion it is enough to show that $\mathcal{H}(Y)$
has a $k$-rational point. Since $C_{\overline{k}}\simeq\mathbb{P}_{\overline{k}}^{1}$,
there exists a quadratic extension $k\subset k'$ such that $C(k')$
is nonempty, so that in particular $\mathcal{H}(Y)_{k'}\simeq\mathbb{P}_{k'}^{2}$.
Let $p$ be a $k'$-rational point of $C_{k'}$. If $p$ is invariant
under the action of the Galois group $\mathrm{Gal}(k'/k)$, then it
corresponds to a $k$-rational point of $C$, hence of $\mathcal{H}(Y)$,
and we are done. Otherwise, its Galois conjugate $\overline{p}$ is
a $k'$-rational point of $C_{k'}$ distinct from $p$, and then the
tangent lines $T_{p}C_{k'}$ and $T_{\overline{p}}C_{k'}$ to $C_{k'}$
at $p$ and $\overline{p}$ respectively intersect each other at unique
point. The latter is thus $\mathrm{Gal}(k'/k)$-invariant, hence corresponds
to a $k$-rational point of $\mathcal{H}(Y)\setminus C$. The existence
of lines with trivial normal bundles defined over $k$ then follows
from the description of $\mathcal{H}(V_{5})$ given in $\S$ \ref{subsec:Hilbert-Lines}. 

The second assertion is an immediate consequence of the facts that
special lines in $Y_{\overline{k}}$ are in one-to-one correspondence
with closed points of the image $D$ of the intersection of $(\mathcal{U}|_{C})_{\overline{k}}$
with the ramification locus of $\upsilon$, and that $\upsilon(\mathcal{U}|_{C})_{\overline{k}}$
coincides with the surface swept out by the tangent lines to $D$. 
\end{proof}
\begin{cor}
\label{cor:FormRational-Divisible}A $k$-form $Y$ of $V_{5}$ is
$k$-rational and the natural map $\mathrm{Pic}(Y)\rightarrow\mathrm{Pic}(Y_{\overline{k}})$
is an isomorphism. 
\end{cor}

\begin{proof}
By Lemma \ref{lem:Forms-General-Line} a), $Y$ contains a line $\ell\simeq\mathbb{P}_{k}^{1}$
with trivial normal bundle. The surface $Z$ in $Y_{\overline{k}}$
swept out by lines intersecting $\ell_{\overline{k}}$ is defined
over $k$. As in $\S$ \ref{subsec:Iskov-link}, the composition of
the blow-up $q:\tilde{Y}\rightarrow Y$ of $\ell$ with exceptional
divisor $Q_{0}'\simeq\mathbb{P}_{k}^{1}\times\mathbb{P}_{k}^{1}$
followed by the contraction $q':\tilde{Y}\rightarrow Q$ of the proper
transform of $Z'$ of $Z$ yields a birational map $Y\dashrightarrow Q$
defined over $k$ onto a smooth quadric $Q\subset\mathbb{P}_{k}^{4}$,
which maps $Q_{0}'$ onto a hyperplane section $Q_{0}$ of $Q$. Since
$Q$ contains $k$-rational points, it is $k$-rational. The natural
map $\mathrm{Pic}(Y)\rightarrow\mathrm{Pic}(Y_{\overline{k}})$ is
an isomorphism if and only $-K_{Y}$ is divisible in $\mathrm{Pic}(Y)$.
But since $-K_{Q}\sim3Q_{0}$ and ${q'}^{*}Q_{0}=Q_{0}'+Z'$, we deduce
from the ramification formula for $q$ and $q'$ that 
\[
-K_{Y}\sim q_{*}(-K_{\tilde{Y}})\sim q_{*}(-{q'}^{*}K_{Q}-Z')\sim q_{*}(3Q_{0}'+2Z')\sim2Z.
\]
\end{proof}

\subsection{Varieties of trisecant lines to Veronese surfaces in $\mathbb{P}^{4}$}

In this subsection, we review a third classical construction of $V_{5}$
as the variety of trisecant line to the Veronese surface in $\mathbb{P}_{\overline{k}}^{4}$
which, when performed over $k$ gives rise, depending on the choices
made, to nontrivial $k$-forms of $V_{5}$ . 

Let $V$ be a $k$-vector space of dimension $3$, let $f\in\mathrm{Sym}^{2}V^{*}$
be a homogeneous form defining a smooth conic $Q\subset\mathbb{P}(V)$,
and let $W=\mathrm{Sym}^{2}(V^{*}/\langle f\rangle)$. Recall that
a closed subscheme $Z\subset\mathbb{P}(V^{*})$ defined over $k$
is called \emph{apolar} to $Q$ if the class of $[f]\in\mathbb{P}(\mathrm{Sym}^{2}V^{*})$
lies in the linear span of the image of $Z$ by the second Veronese
embedding $v_{2}:\mathbb{P}(V^{*})\hookrightarrow\mathbb{P}(\mathrm{Sym}^{2}V^{*})$.
When $Z=\{[\ell_{1}],[\ell_{2}],[\ell_{3}]\}\subset\mathbb{P}(V^{*})$
consists of three distinct $k$-rational points, this says equivalently
that there are scalars $\lambda_{i}\in k$, $i=1,2,3$, such that
$f=\lambda_{1}\ell_{1}^{2}+\lambda_{2}\ell_{2}^{2}+\lambda_{3}\ell_{3}^{2}$.
We denote by $\mathrm{VSP}(f)$ the variety of closed subschemes of
length $3$ of $\mathbb{P}(V^{*})$ which are apolar to $Q$.

Let $\pi_{[f]}:\mathbb{P}(\mathrm{Sym}^{2}V^{*})\dashrightarrow\mathbb{P}(W)$
be the projection from $[f]$. Since $Q$ is smooth, the composition
$\pi_{[f]}\circ v_{2}$ is a closed embedding of $\mathbb{P}(V^{*})$,
whose image is a Veronese surface $X$ in $\mathbb{P}(W)$. Since
$X$ is defined over $k$, the closed subscheme $Y$ of the Grassmannian
$G(1,\mathbb{P}(W_{\overline{k}}))$ of lines in $\mathbb{P}(W_{\overline{k}})$
consisting of trisecant lines to $X_{\overline{k}}$, that is, lines
$\ell$ in $\mathbb{P}(W_{\overline{k}})$ such that $\ell\cdot X_{\overline{k}}$
is a closed subscheme of length $3$ of $\ell$,  is defined over
$k$, and we obtain an identification between $\mathrm{VSP}(f)$ and
$Y$. 
\begin{prop}
$($see \cite[\S 2.3]{Il94}$)$ With the notation above, $Y=\mathrm{VSP}(f)$
is a $k$-form of $V_{5}$. Furthermore, the stratification of $Y_{\overline{k}}$
by the types of the corresponding closed subschemes of length $3$
of $\mathbb{P}(V^{*})$ is related to that of $V_{5}$ as a quasi-homogenous
space of $\mathrm{PGL}_{2}(\overline{k})$ given in $\S$ \ref{subsec:QuasiHom}
as follows:

\begin{enumerate}[label=$($\roman*$)$]

\item Points of the open orbit $O$ correspond to reduced subschemes\textup{
$\{[\ell_{1}],[\ell_{2}],[\ell_{3}]\}$ }apolar to\textup{ $Q_{\overline{k}}$
}such that none of the\textup{ $[\ell_{i}]$ }belongs to the dual
conic\textup{ $Q_{\overline{k}}^{*}\subset\mathbb{P}(V_{\overline{k}}^{*})$
}of \textup{$Q_{\overline{k}}$.}

\item Points in the $2$-dimensional orbit $S_{2}$ correspond to
non-reduced subschemes $\{2[\ell_{1}],[\ell_{2}]\}$ where $[\ell_{1}]$
belongs to $Q_{\overline{k}}^{*}$ and $[\ell_{2}]$ is a point of
the tangent line $T_{[\ell_{1}]}Q_{\overline{k}}^{*}$ to $Q_{\overline{k}}^{*}$
at $[\ell_{1}]$ distinct from $[\ell_{1}]$. 

\item Points in the $1$-dimensional orbit $C_{6}$ correspond to
non-reduced subschemes $\{3[\ell]\}$ where $[\ell]$ is a point of
$Q_{\overline{k}}^{*}$. 

\end{enumerate}
\end{prop}

The Hilbert scheme $\mathcal{H}(Y)$ of lines on a $k$-form $Y=\mathrm{VSP}(f)\subset G(1,\mathbb{P}(W))$
can be explicitly described as follows. Viewing $G(1,\mathbb{P}(W))$
as a closed subscheme of $\mathbb{P}(\Lambda^{2}W)$ via the Plücker
embedding, $\mathcal{H}(Y)$ is a closed subscheme of the Hilbert
scheme $\mathcal{H}(G(1,\mathbb{P}(W)))$ of lines on $G(1,\mathbb{P}(W))$.
The latter is isomorphic to the flag variety $\mathcal{F}(0,2,\mathbb{P}(W))\subset\mathbb{P}(W)\times G(2,\mathbb{P}(W))$
of pairs $(x,P)$ consisting of a point $x\in\mathbb{P}(W)$ and a
$2$-dimensional linear space $P\subset\mathbb{P}(W)$ containing
it. 
\begin{prop}
\label{prop:VSP-lines}$($see e.g. \cite[\S 1.2]{Il94-2}$)$ Let
$f\in\mathrm{Sym}^{2}V^{*}$ be a homogeneous form defining a smooth
conic $Q\subset\mathbb{P}(V)$, let $W=\mathrm{Sym}^{2}(V^{*}/\langle f\rangle)$,
and let $\mathrm{VSP}(f)\subset G(1,\mathbb{P}(W))$ be the variety
of trisecant lines to the Veronese surface $X=\pi_{[f]}\circ v_{2}(\mathbb{P}(V^{*}))$. 

\begin{enumerate}[label=$($\roman*$)$]

\item The projection $\mathrm{pr}_{1}:\mathcal{F}(0,2,\mathbb{P}(W))\rightarrow\mathbb{P}(W)$
restricts to a closed embedding $\mathcal{H}(\mathrm{VSP}(f))\hookrightarrow\mathbb{P}(W)$
whose image is equal to $X\simeq\mathbb{P}(V^{*})$.

\item The image of the smooth conic $C\subset\mathcal{H}(\mathrm{VSP}(f))$
parametrizing special lines in $\mathrm{VSP}(f)$ coincides with the
conic $Q^{*}\subset\mathbb{P}(V^{*})$ dual to $Q\subset\mathbb{P}(V)$.

\item For every line $\ell$ in $\mathrm{VSP}(f)_{\overline{k}}$
defined by a point $x\in\mathbb{P}(V_{\overline{k}}^{*})$, the set
of lines in $\mathrm{VSP}(f)_{\overline{k}}$ which intersect $\ell$
is parametrized by the conjugate line of $x$ with respect to $Q_{\overline{k}}^{*}$. 

\end{enumerate}
\end{prop}

As a consequence, we obtain the following characterization of which
$k$-forms $Y=\mathrm{VSP}(f)$ of $V_{5}$ contain special lines
defined over $k$. 
\begin{cor}
The $k$-form $Y=\mathrm{VSP}(f)$ of $V_{5}$ contains a special
line defined over $k$ if and only if the conic $Q=V(f)\subset\mathbb{P}(V)$
has a $k$-rational point.
\end{cor}

\begin{proof}
By combining Lemma \ref{lem:Existence-of-specialLines} b) and Proposition
\ref{prop:VSP-lines} $($ii$)$, we deduce that $Y$ contains a special
line defined over $k$ if and only if the conic $Q^{*}\subset\mathbb{P}(V^{*})$
dual to $Q$ has a $k$-rational point, hence if and only if $Q$
has a $k$-rational point. 
\end{proof}
\begin{example}
\label{exa:form-without-special-line} Let $k=\mathbb{C}(s,t)$ with
$s,t$ algebraically independent over $\mathbb{C}$, and let $f=x^{2}+sy^{2}+tz^{2}\in k[x,y,z]$.
Since the conic $Q=V(f)\subset\mathbb{P}_{k}^{2}=\mathrm{Proj}_{k}(k[x,y,z])$
has no $k$-rational point, the variety $Y=\mathrm{VSP}(f)$ is a
$k$-form of $V_{5}$ which does not contain any special line defined
over $k$. 
\end{example}

\section{\label{sec:3}Double projection from a rational point }

Let $V_{5}\hookrightarrow\mathbb{P}_{\overline{k}}^{6}$ be embedded
by its half-anticanonical complete linear system. For every closed
point $y\in V_{5}$, the linear system $|\mathcal{O}_{V_{5}}(1)\varotimes\mathfrak{m}_{y}^{2}|$
of hyperplane sections of $V_{5}$ which are singular at $y$ defines
a rational map $\psi_{y}:V_{5}\dashrightarrow\mathbb{P}_{\overline{k}}^{2}$
called the \emph{double projection from} $y$, whose description is
classical knowledge in the birational geometry of threefolds \cite{FuNa89-2,KoPr92}.
In this section, inspired by \cite{Ta89}, we give an explicit ``reverse
construction'' of these maps, valid over any field $k$ of characteristic
zero, formulated in terms of Sarkisov links performed from certain
locally trivial $\mathbb{P}^{1}$-bundles over $\mathbb{P}^{2}$ associated
to standard birational quadratic involutions of $\mathbb{P}^{2}$. 

\subsection{\label{subsec:Quadric-involutions}Recollection on standard quadratic
involutions of $\mathbb{P}^{2}$ }

Let $\sigma:\mathbb{P}_{k}^{2}\dashrightarrow\mathbb{P}_{k}^{2}$
be a birational quadratic involution of $\mathbb{P}_{k}^{2}$ and
let $\Gamma_{\sigma}\subset\mathbb{P}_{k}^{2}\times\mathbb{P}_{k}^{2}$
be its graph. Via the two projections $\tau=\mathrm{pr}_{1}:\Gamma_{\sigma}\rightarrow\mathbb{P}_{k}^{2}$
and $\tau'=\mathrm{pr}_{2}:\Gamma_{\sigma}\rightarrow\mathbb{P}_{k}^{2}$,
$\Gamma_{\sigma}$ is canonically identified with the blow-up of the
scheme-theoretic base loci $Z=\mathrm{Bs}(\sigma)$ and $Z'=\mathrm{Bs}(\sigma^{-1})$
respectively. We denote by $e=\tau^{*}Z$ and $e'={\tau'}{}^{*}Z'$
the scheme-theoretic inverse images of $Z$ and $Z'$ respectively
in $\Gamma_{\sigma}$. It is well-known that $Z$ and $Z'$ are local
complete intersection $0$-dimensional closed sub-scheme of length
$3$ of $\mathbb{P}_{k}^{2}$ with one of the following possible structure:

$($Type I$)$ $Z$ (resp. $Z'$) is smooth and its base extension
to $\overline{k}$ consists of three non-collinear points $p_{1},p_{2},p_{3}$
(resp. $p_{1}',p_{2}',p_{3}')$ whose union is defined over $k$.
The surface $\Gamma_{\sigma}$ is smooth, $e_{\overline{k}}=e_{p_{1}}+e_{p_{2}}+e_{p_{3}}$,
$e'_{\overline{k}}=e'_{p_{1}'}+e'_{p_{2}'}+e'_{p_{3}'}$ where the
$e_{p_{i}}$ and $e'_{p_{i}'}$ are $(-1)$-curves. 

$($Type II$)$ $Z$ (resp. $Z'$) consists of the disjoint union
$(p_{1},\mathbf{p}_{2})$ (resp. $(p_{1}',\mathbf{p}_{2}')$) of a
smooth $k$-rational point $p_{1}$ (resp. $p_{1}'$) and a $0$-dimensional
sub-scheme $\mathbf{p}_{2}$ (resp. $\mathbf{p}_{2}')$ of length
$2$ locally isomorphic to $V(x,y^{2})$ and supported at a $k$-rational
point $p_{2}$ (resp. $p_{2}'$). We have $e=e_{p_{1}}+2e_{p_{2}}$,
$e'=e'_{p_{1}'}+2e'_{p_{2}'}$ where $e_{p_{1}}$ and $e_{p_{2}}$
(resp. $e_{p_{1}'}'$ and $e_{p_{2}'}'$) are smooth $k$-rational
curves with self-intersections $-1$ and $-\frac{1}{2}$ respectively,
and $\Gamma_{\sigma}$ has a unique $A_{1}$-singularity at the $k$-rational
point $e_{p_{2}}\cap e_{p_{2}'}'$. 

$($Type III$)$ $Z$ (resp. $Z'$) is supported on a unique $k$-rational
point $p$ (resp. $p'$) and is locally isomorphic to $V(y^{3},x-y^{2})$.
We have $e=3e_{p}$ and $e'=3e_{p'}'$ where $e_{p}$ and $e_{p'}'$
are smooth $k$-rational curves with self-intersection $-\frac{1}{3}$,
and $\Gamma_{\sigma}$ has a unique $A_{2}$-singularity at the $k$-rational
point $e_{p}\cap e'_{p'}$. \\

The following lemma whose proof is left to the reader records for
later use some additional basic properties of the surfaces $\Gamma_{\sigma}$. 
\begin{lem}
\label{lem:Graph-Properties}With the notation above, the following
hold:

a) The canonical divisor $K_{\Gamma_{\sigma}}$ is linearly equivalent
to $-e-e'$.

b) The pull-back by $\tau:\Gamma_{\sigma}\rightarrow\mathbb{P}_{k}^{2}$
$($resp. $\tau':\Gamma_{\sigma}\rightarrow\mathbb{P}_{k}^{2}$$)$
of a general line $\ell\simeq\mathbb{P}_{k}^{1}$ in $\mathbb{P}_{k}^{2}$
is $\mathbb{Q}$-linearly equivalent to $\frac{1}{3}(2e+e')$ $($resp.
$\frac{1}{3}(e+2e')$$)$. 
\end{lem}

\subsection{\label{subsec:ProjectiveBundle}Locally trivial $\mathbb{P}^{1}$-bundles
over $\mathbb{P}^{2}$ associated to quadratic involutions }

Let $\sigma:\mathbb{P}_{k}^{2}\dashrightarrow\mathbb{P}_{k}^{2}$
be a birational quadratic involution of $\mathbb{P}_{k}^{2}$ with
graph $\Gamma_{\sigma}\subset\mathbb{P}_{k}^{2}\times\mathbb{P}_{k}^{2}$,
and let $\mathcal{I}_{Z}\subset\mathcal{O}_{\mathbb{P}_{k}^{2}}$
be the ideal sheaf of its scheme-theoretic base locus $Z$. 
\begin{lem}
\label{lem:Vector-bundle}There exists a locally free sheaf $\mathcal{E}$
of rank $2$ on $\mathbb{P}_{k}^{2}$ and an exact sequence 
\begin{equation}
0\rightarrow\mathcal{O}_{\mathbb{P}_{k}^{2}}(-1)\stackrel{s}{\longrightarrow}\mathcal{E}\longrightarrow\mathcal{I}_{Z}\rightarrow0.\label{eq:Bundle-sequence}
\end{equation}
\end{lem}

\begin{proof}
Since $Z$ is a local complete intersection of codimension $2$ in
$\mathbb{P}_{k}^{2}$, the existence of $\mathcal{E}$ with the required
properties follows from Serre correspondence \cite{Se61}. More precisely,
the local-to-global spectral sequence 
\[
E_{2}^{p,q}=H^{P}(\mathbb{P}_{k}^{2},\mathcal{E}xt^{q}(\mathcal{I}_{Z},\mathcal{O}_{\mathbb{P}_{k}^{2}}(-1))\Rightarrow\mathrm{Ext}^{p+q}(\mathcal{I}_{Z},\mathcal{O}_{\mathbb{P}_{k}^{2}}(-1))
\]
degenerates to a long exact sequence 
\[
0\rightarrow H^{1}(\mathbb{P}_{k}^{2},\mathcal{O}_{\mathbb{P}_{k}^{2}}(-1))\rightarrow\mathrm{Ext}^{1}(\mathcal{I}_{Z},\mathcal{O}_{\mathbb{P}_{k}^{2}}(-1))\rightarrow H^{0}(Z,\det\mathcal{N}_{Z/\mathbb{P}_{k}^{2}}\varotimes\mathcal{O}_{Z}(-1))\rightarrow H^{2}(\mathbb{P}_{k}^{2},\mathcal{O}_{\mathbb{P}_{k}^{2}}(-1)).
\]
Since $H^{i}(\mathbb{P}_{k}^{2},\mathcal{O}_{\mathbb{P}_{k}^{2}}(-1))=0$
for $i=1,2$ and $Z$ is $0$-dimensional, this sequence provides
an isomorphism 
\[
\mathrm{Ext}^{1}(\mathcal{I}_{Z},\mathcal{O}_{\mathbb{P}_{k}^{2}}(-1))\simeq H^{0}(Z,\det\mathcal{N}_{Z/\mathbb{P}_{k}^{2}}\varotimes\mathcal{O}_{Z}(-1))\simeq H^{0}(Z,\mathcal{O}_{Z}).
\]
The extension corresponding via this isomorphism to the constant section
$1\in H^{0}(Z,\mathcal{O}_{Z})$ has the desired property. 
\end{proof}
Let $\pi:\mathbb{P}(\mathcal{E})=\mathrm{Proj}(\mathrm{Sym}^{\cdot}\mathcal{E})\rightarrow\mathbb{P}_{k}^{2}$
be the locally trivial $\mathbb{P}^{1}$-bundle associated with the
locally free sheaf $\mathcal{E}$ of rank $2$ as in Lemma \ref{lem:Vector-bundle}.
Since $\det\mathcal{E}\simeq\mathcal{O}_{\mathbb{P}_{k}^{2}}(-1)$,
the canonical sheaf $\omega_{\mathbb{P}(\mathcal{E})}$ of $\mathbb{P}(\mathcal{E})$
is isomorphic to 
\begin{equation}
\mathcal{O}_{\mathbb{P}(\mathcal{E})}(-2)\varotimes\pi^{*}\det\mathcal{E}\varotimes\pi^{*}\omega_{\mathbb{P}_{k}^{2}}\simeq\mathcal{O}_{\mathbb{P}(\mathcal{E})}(-2)\varotimes\pi^{*}\mathcal{O}_{\mathbb{P}_{k}^{2}}(-4).\label{eq:canonical-bundle}
\end{equation}
The surjection $\mathcal{E}\rightarrow\mathcal{I}_{Z}\rightarrow0$
defines a closed embedding $\mathbb{P}(\mathcal{I}_{Z})=\mathrm{Proj}(\mathrm{Sym}^{\cdot}\mathcal{I}_{Z})\hookrightarrow\mathbb{P}(\mathcal{E})$.
Since $Z$ is a local complete intersection, the canonical homomorphism
of graded $\mathcal{O}_{\mathbb{P}_{k}^{2}}$-algebras $\mathrm{Sym}^{\cdot}\mathcal{I}_{Z}\rightarrow\mathcal{R}(\mathcal{I}_{Z})=\bigoplus_{n\geq0}\mathcal{I}_{Z}\cdot t^{n}$
is an isomorphism \cite[Théorème 1]{Mi64}, and it follows that the
restriction $\pi:\mathbb{P}(\mathcal{I}_{Z})\rightarrow\mathbb{P}_{k}^{2}$
is isomorphic to the blow-up $\tau:\Gamma_{\sigma}\rightarrow\mathbb{P}_{k}^{2}$
of $Z$. We can thus identify from now on $\Gamma_{\sigma}$ with
the closed sub-scheme $\mathbb{P}(\mathcal{I}_{Z})$ of $\mathbb{P}(\mathcal{E})$.
The composition of $\pi^{*}s:\pi^{*}\mathcal{O}_{\mathbb{P}_{k}^{1}}(-1)\rightarrow\pi^{*}\mathcal{E}$
with the canonical surjection $\pi^{*}\mathcal{E}\rightarrow\mathcal{O}_{\mathbb{P}(\mathcal{E})}(1)$
defines a global section 
\[
\overline{s}\in\mathrm{Hom}_{\mathbb{P}(\mathcal{E})}(\pi^{*}\mathcal{O}_{\mathbb{P}_{k}^{1}}(-1),\mathcal{O}_{\mathbb{P}(\mathcal{E})}(1))\simeq H^{0}(\mathbb{P}(\mathcal{E}),\mathcal{O}_{\mathbb{P}(\mathcal{E})}(1)\varotimes\pi^{*}\mathcal{O}_{\mathbb{P}_{k}^{2}}(1))
\]
whose zero locus $V(\overline{s})$ coincides with $\Gamma_{\sigma}$.
\\

Letting $\xi$ and $A$ be the classes in the divisor class group
of $\mathbb{P}(\mathcal{E})$ of $\mathcal{O}_{\mathbb{P}(\mathcal{E})}(1)$
and of the inverse image of a general line in $\mathbb{P}_{k}^{2}$
by $\pi$ respectively, we have $\Gamma_{\sigma}=V(\overline{s})\sim\xi+A$.
On the other hand, it follows from (\ref{eq:canonical-bundle}) that
the canonical divisor $K_{\mathbb{P}(\mathcal{E})}$ of $\mathbb{P}(\mathcal{E})$
is linearly equivalent to $-2(\Gamma_{\sigma}+A)\sim-2(\xi+2A)$.
Since $c_{1}(\mathcal{E})=-1$ and $c_{2}(\mathcal{E})=3$ by construction,
we derive the following numerical information:
\begin{equation}
\begin{array}{cccc}
K_{\mathbb{P}(\mathcal{E})}^{3}=-32,\quad & K_{\mathbb{P}(\mathcal{E})}^{2}\cdot\Gamma_{\sigma}=4,\quad & K_{\mathbb{P}(\mathcal{E})}\cdot\Gamma_{\sigma}^{2}=2,\quad & \Gamma_{\sigma}^{3}=-2.\end{array}\label{eq:(triple intersections)}
\end{equation}

\subsection{\label{subsec:Sarkisov-DProj}Construction of Sarkisov links}
\begin{prop}
\label{prop:DoubleProj-SarkisovLink} Let $\pi:\mathbb{P}(\mathcal{E})\rightarrow\mathbb{P}_{k}^{2}$
be the $\mathbb{P}^{1}$-bundle associated to a quadratic involution
$\sigma:\mathbb{P}_{k}^{2}\dashrightarrow\mathbb{P}_{k}^{2}$ with
graph $\Gamma_{\sigma}\subset\mathbb{P}(\mathcal{E})$ as in \S \ref{subsec:ProjectiveBundle}.
Then there exists a $k$-form $Y$ of $V_{5}$ and a Sarkisov link
\[\xymatrix@R=1em@C=2em{ & \Gamma_\sigma  \ar@{^{(}->}[rr] \ar[ddr]_{\tau} & & \mathbb{P}(\mathcal{E}) \ar[ddl]_{\pi} \ar[dr] \ar@{-->}[rr]^{\varphi} & & X^{+} \ar[dl] \ar[ddr]^{\theta} & & \Gamma_{\sigma}^{+}\ar@{_{(}->}[ll] \ar[ddr]^{\theta|_{\Gamma^{+}}} \\ & & & & X_0 & & & &\\ &   & \mathbb{P}^2_{k} & & & & Y  & & \{y\}  \ar@{_{(}->}[ll]} \]
where:

\begin{enumerate}[label=$($\roman*$)$]

\item $\varphi:\mathbb{P}(\mathcal{E})\dashrightarrow X^{+}$ is
a flop whose flopping locus coincides with the support of $e'\subset\Gamma_{\sigma}$,

\item $\Gamma_{\sigma}^{+}\simeq\mathbb{P}_{k}^{2}$ is the proper
transform of $\Gamma_{\sigma}$ and $\theta|_{\Gamma_{\sigma}}:\Gamma_{\sigma}\dashrightarrow\Gamma_{\sigma}^{+}$
is the contraction of $e'$,

\item $\theta:X^{+}\rightarrow Y$ is the divisorial contraction
of $\Gamma_{\sigma}^{+}$ to a smooth $k$-rational point $y\in Y$,

\item The support of the image in $Y$ of the flopped locus $e^{+}\subset X^{+}$
of $\varphi$ coincides with the union of the lines in $Y_{\overline{k}}$
passing through $y$. 

\end{enumerate}
\end{prop}

\begin{proof}
We denote $\mathbb{P}(\mathcal{E})$ and $\Gamma_{\sigma}\subset\mathbb{P}(\mathcal{E})$
simply by $X$ and $\Gamma$. Since $-K_{X}\sim2\xi+4A\sim2\Gamma+2A$
we see that $-K_{X}$ is nef and that any irreducible curve $C\subset X_{\overline{k}}$
such that $-K_{X_{\overline{k}}}\cdot C\leq0$ is contained in $\Gamma_{\overline{k}}$.
By the adjunction formula and Lemma \ref{lem:Graph-Properties} a),
we have 
\begin{align*}
\Gamma^{2}=-K_{\Gamma}-2A\cdot\Gamma & \sim_{\mathbb{Q}}\frac{1}{3}(-e+e')\quad\textrm{and}\quad\Gamma^{2}-K_{\Gamma}\sim_{\mathbb{Q}}\frac{2}{3}(e+2e'),
\end{align*}
which implies by adjunction again that 
\begin{align*}
-K_{X_{\overline{k}}}\cdot C & =(\Gamma_{\overline{k}}^{2}-K_{\Gamma_{\overline{k}}})\cdot C=\frac{2}{3}(e{}_{\overline{k}}+2e_{\overline{k}}')\cdot C.
\end{align*}
Since $e+2e'$ is $\tau$-ample and $\tau'$-numerically trivial by
virtue of Lemma \ref{lem:Graph-Properties} b), we conclude that the
irreducible curves $C\subset X_{\overline{k}}$ such that $-K_{X_{\overline{k}}}\cdot C=0$
are precisely the irreducible components of the exceptional locus
$e'_{\overline{k}}$ of $\tau'_{\overline{k}}:\Gamma_{\overline{k}}\rightarrow\mathbb{P}_{\overline{k}}^{2}$
(see $\S$ \ref{subsec:Quadric-involutions} for the notation). 

Let $\varphi:X\dashrightarrow X^{+}$ be the flop of the union of
the irreducible components of $e'_{\overline{k}}$. Since the union
of these components is defined over $k$, so is the union $e^{+}$
of the flopped curves of $\varphi$, and so $X^{+}$ is a smooth threefold
defined over $k$ and $\varphi$ is a birational map defined over
$k$ restricting to an isomorphism between $X\setminus e'$ and $X^{+}\setminus e^{+}$.
Let $\Gamma^{+}$ and $A^{+}$ be the proper transforms in $X^{+}$
of $\Gamma$ and $A$ respectively. By construction, the restriction
$\varphi|_{\Gamma}:\Gamma\dashrightarrow\Gamma^{+}$ coincide with
$\tau':\Gamma\rightarrow\mathbb{P}_{k}^{2}$. Since $K_{X^{+}}\sim-2(\Gamma^{+}+A^{+})$,
we deduce from the adjunction formula that 
\[
-(\Gamma^{+})^{2}=K_{\Gamma^{+}}+2A^{+}\cdot\Gamma^{+}=\tau'_{*}(K_{\Gamma}+2A\cdot\Gamma)=\frac{1}{3}\tau'_{*}e,
\]
which is linearly equivalent to a line $\ell\simeq\mathbb{P}_{k}^{1}$
in $\Gamma^{+}\simeq\mathbb{P}_{k}^{2}$. The normal bundle $\mathcal{N}_{\Gamma^{+}/X^{+}}$
of $\Gamma^{+}$ in $X^{+}$ is thus isomorphic to $\mathcal{O}_{\mathbb{P}_{k}^{2}}(-1)$. 

The divisor class group of $X^{+}$ is freely generated by $\Gamma^{+}$
and $A^{+}$. The Mori cone $\mathrm{NE}(X^{+})$ is spanned by two
extremal rays: one $R_{1}$ corresponding to the flopped curves of
$\varphi$ and a second one $R_{2}$ which is $K_{X^{+}}$-negative.
Since $\Gamma^{+}\cdot C\geq0$ for any irreducible curve not contained
in $\Gamma^{+}$, including thus the irreducible components of $e^{+}$,
whereas $\Gamma^{+}\cdot\ell=-1$ for any line $\ell\simeq\mathbb{P}_{k}^{1}$
in $\Gamma^{+}\simeq\mathbb{P}_{k}^{2}$, it follows that $R_{2}$
is generated by the class of $\ell$. The extremal contraction associated
to $R_{2}$ is thus the divisorial contraction $\theta:X^{+}\rightarrow Y$
of $\Gamma^{+}\simeq\mathbb{P}_{k}^{2}$ to a smooth $k$-rational
point $p$ of a smooth projective threefold $Y$. 

Since $-K_{X^{+}}=2(\Gamma^{+}+A^{+})$, we conclude that the image
$H$ of $A^{+}$ by $\theta$ is an ample divisor on $Y$ generating
the divisor class group of $Y$ and such that $-K_{Y}=2H$. Furthermore,
we have 
\begin{align*}
K_{Y}^{3} & =\theta^{*}(K_{Y})^{3}=(K_{X^{+}}-2\Gamma^{+})^{3}\\
 & =K_{X^{+}}^{3}-6K_{X^{+}}^{2}\cdot\Gamma^{+}+12K_{X^{+}}\cdot(\Gamma^{+})^{2}-8(\Gamma^{+})^{3}\\
 & =K_{X}^{3}-6K_{X}^{2}\cdot\Gamma+12K_{X}\cdot\Gamma-8
\end{align*}
so that $K_{Y}^{3}=-40$ by (\ref{eq:(triple intersections)}). Altogether,
this shows $Y_{\overline{k}}$ is a smooth Fano threefold of Picard
number $1$, index $2$ and degree $d=H^{3}=5$, hence is isomorphic
to $V_{5}$. 

The fact that the support of the image in $Y$ of $e^{+}$ coincides
with the union of the lines in $Y_{\overline{k}}$ passing through
$p$ is clear by construction. 
\end{proof}
By construction, the proper transform by the reverse composition $\psi_{y}=\pi\circ\varphi^{-1}\circ\theta^{-1}:Y\dashrightarrow\mathbb{P}(\mathcal{E})\stackrel{\pi}{\rightarrow}\mathbb{P}_{k}^{2}$
of the complete linear system of lines in $\mathbb{P}_{k}^{2}$ consists
of divisors $H$ on $Y$ singular at $y$ and such that $-K_{Y}=2H$.
This shows that $\psi_{y,\overline{k}}:Y_{\overline{k}}\simeq V_{5}\dashrightarrow\mathbb{P}_{\overline{k}}^{2}$
coincides with the double projection from the point $y$. Conversely,
given any $k$-form $Y$ of $V_{5}$, Corollary \ref{cor:FormRational-Divisible}
ensures that $-K_{Y}$ is divisible, equal to $2H$ for some ample
divisor $H$ on $Y$. So given any $k$-rational point $y\in Y$,
the double projection $\psi_{y}:Y\dashrightarrow\mathbb{P}_{k}^{2}$
from $y$ is defined over $k$, given by the linear system $|\mathcal{O}_{Y}(H)\varotimes\mathfrak{m}_{y}^{2}|$,
and it coincides with the composition $\pi\circ\varphi^{-1}\circ\theta^{-1}:Y\dashrightarrow\mathbb{P}(\mathcal{E})\stackrel{\pi}{\rightarrow}\mathbb{P}_{k}^{2}$
for a suitable quadratic birational involution $\sigma$ of $\mathbb{P}_{k}^{2}$. 

\section{\label{sec:4}Application : cylinders in forms of the quintic del
Pezzo threefold}
\begin{thm}
Let $Y$ be a $k$-form of $V_{5}$, let $y\in Y$ be a $k$-rational
point and let $C$ be the union of the lines in $Y_{\overline{k}}$
passing through $y$. Then $Y\setminus C$ has the structure of a
Zariski locally trivial $\mathbb{A}^{1}$-bundle $\rho:Y\setminus C\rightarrow\mathbb{P}_{k}^{2}\setminus Z$
over the complement of a closed sub-scheme $Z\subset\mathbb{P}_{k}^{2}$
of length $3$ with as many irreducible geometric components as $C$. 
\end{thm}

\begin{proof}
Indeed, the birational map $\xi=\theta\circ\varphi:\mathbb{P}(\mathcal{E})\dashrightarrow Y$
constructed in Proposition \ref{prop:DoubleProj-SarkisovLink} restricts
to an isomorphism between $Y\setminus C$ and the complement of the
proper transform $\Gamma$ in $\mathbb{P}(\mathcal{E})$ of the exceptional
divisor of the blow-up of $Y$ at $y$. On the other hand, it follows
from the construction of $\mathcal{E}$ in $\S$ \ref{subsec:ProjectiveBundle}
that $\Gamma$ is the graph of a quadratic birational involution $\sigma$
of $\mathbb{P}_{k}^{2}$ and that the restriction of $\pi:\mathbb{P}(\mathcal{E})\rightarrow\mathbb{P}_{k}^{2}$
to the complement of $\Gamma$ is a locally trivial $\mathbb{A}^{1}$-bundle
over the complement $\mathbb{P}_{k}^{2}\setminus Z$ of the base locus
$Z$ of $\sigma$. 
\end{proof}
\begin{thm}
Every $k$-form $Y$ of $V_{5}$ contains a Zariski open $\mathbb{A}_{k}^{2}$-cylinder.
Furthermore, $Y$ contains $\mathbb{A}_{k}^{3}$ as a Zariski open
subset if and only if it contains a special line defined over $k$. 
\end{thm}

\begin{proof}
By Lemma \ref{lem:Forms-General-Line}, $Y$ contains a line $\ell\simeq\mathbb{P}_{k}^{1}$
with trivial normal bundle. Projecting from $\ell$ as in the proof
of Corollary \ref{cor:FormRational-Divisible}, we obtain a birational
map $Y\dashrightarrow Q$ onto a smooth quadric $Q\subset\mathbb{P}_{k}^{4}$
defined over $k$, which restricts to an isomorphism between the complement
of the surface $Z\subset Y$ defined over $k$ swept out by the lines
in $Y_{\overline{k}}$ intersecting $\ell_{\overline{k}}$ and the
complement of a hyperplane section $Q_{0}\simeq\mathbb{P}_{k}^{1}\times\mathbb{P}_{k}^{1}$
of $Q$. The complement $Q\setminus Q_{0}$ is isomorphic to the smooth
affine quadric $\tilde{Q}_{0}=\{xv-yu=1\}\subset\mathbb{A}_{k}^{4}$,
which contains for instance the principal open $\mathbb{A}_{k}^{2}$-cylinder
$\tilde{Q}_{0,x}\simeq\mathrm{Spec}(k[x^{\pm1}][y,u])$. 

By Lemma \ref{lem:Existence-of-specialLines} b), $Y$ contains a
special line if and only it has a $k$-rational point $p$ through
which there exists at most two lines in $Y_{\overline{k}}$. Letting
$C$ be the union of these lines, the previous theorem implies that
$\rho:Y\setminus C\rightarrow\mathbb{P}_{k}^{2}\setminus Z$ is a
Zariski locally trivial $\mathbb{A}^{1}$-bundle over the complement
of a $0$-dimensional closed subset $Z$ defined over $k$ whose support
consists of at most two points. Letting $L\simeq\mathbb{P}_{k}^{1}$
be any line in $\mathbb{P}_{k}^{2}$ containing the support of $Z$,
the restriction of $\rho$ over $\mathbb{P}_{k}^{2}\setminus L\simeq\mathbb{A}_{k}^{2}$
is then a trivial $\mathbb{A}^{1}$-bundle, so that $Y\setminus\rho^{-1}(L)$
is a Zariski open subset of $Y$ isomorphic to $\mathbb{A}_{k}^{3}$. 

Conversely, suppose that $Y$ contains a Zariski open subset $U\simeq\mathbb{A}_{k}^{3}$.
Since $U$ is affine, $D=Y\setminus U$ has pure codimension $1$
in $Y$. Since $Y_{\overline{k}}\simeq V_{5}$ and $Y_{\overline{k}}\setminus D_{\overline{k}}\simeq U_{\overline{k}}\simeq\mathbb{A}_{\overline{k}}^{3}$,
according to \cite[Corollary 1.2, b]{FuNa89} and \cite{FuNa89-2},
we have the following alternative for the divisor $D_{\overline{k}}$
in $Y_{\overline{k}}$: 

1) $D_{\overline{k}}$ is a normal del Pezzo surface of degree $5$
with a unique singular point $q$ of type $A_{4}$,

2) $D_{\overline{k}}$ is a non-normal del Pezzo surface of degree
$5$ whose singular locus is a special line $\ell$ in $Y_{\overline{k}}$,
and $D_{\overline{k}}$ is swept out by lines in $V_{5}$ which intersect
$\ell$. 

In the first case, there exists a unique special line $\ell$ in $Y_{\overline{k}}$
passing through $q$, which is then automatically contained in $D_{\overline{k}}$.
Since $D$ is defined over $k$ and $q$ is the unique singular point
of $D_{\overline{k}}$, $q$ corresponds to a $k$-rational point
of $D$, so that $\ell$ is actually a special line in $Y$ defined
over $k$. In the second case, since $\ell$ is the singular locus
of $D_{\overline{k}}$, it is invariant under the action of the Galois
group $\mathrm{Gal}(\overline{k}/k)$, hence is defined over $k$.
So in both cases, $Y$ contains a special line defined over $k$. 
\end{proof}
\begin{cor}
\label{cor:A3-overC1-field}Every form of $V_{5}$ defined over a
$C_{1}$-field $k$ contains $\mathbb{A}_{k}^{3}$ as a Zariski open
subset. 
\end{cor}

\begin{proof}
This follows from Lemma \ref{lem:Existence-of-specialLines} b) and
the previous theorem since every conic over a $C_{1}$-field has a
rational point. 
\end{proof}
\begin{example}
The $\mathbb{C}(s,t)$-form $Y=\mathrm{VSP}(f)$ of $V_{5}$ associated
to the quadratic form $f=x^{2}+sy^{2}+tz^{2}\in\mathbb{C}(s,t)[x,y,z]$
considered in Example \ref{exa:form-without-special-line} does not
contain any special line defined over $\mathbb{C}(s,t)$, hence does
not contain $\mathbb{A}_{\mathbb{C}(s,t)}^{3}$ as a Zariski open
subset. 
\end{example}

\bibliographystyle{amsplain} 

\end{document}